\documentclass[11pt]{article}

\usepackage[letterpaper,left=1.25in,right=1.25in,top=1in,bottom=1in]{geometry}
\usepackage{amsmath,amssymb,amsthm}
\usepackage{setspace}
\usepackage{microtype}
\usepackage{needspace}
\usepackage{hyperref}

\allowdisplaybreaks
\newtheorem{theorem}{Theorem}[section]
\newtheorem{conjecture}[theorem]{Conjecture}
\newtheorem{lemma}[theorem]{Lemma}
\newtheorem{claim}{Claim}[section]

\hypersetup{
  hidelinks,
  pdfauthor={Xinheng Lin},
  pdftitle={An Asymptotically Tight t log t Bound for k-Connected Subgraphs in Dense Kt-Minor-Free Graphs},
  pdfsubject={Graph minors and vertex connectivity},
  pdfkeywords={Hadwiger's conjecture, graph minors, vertex connectivity, dense subgraphs, random graphs}
}

\title{An Asymptotically Tight \texorpdfstring{$t\log t$}{t log t} Bound for
  \texorpdfstring{$k$}{k}-Connected Subgraphs in Dense
  \texorpdfstring{$K_t$}{Kt}-Minor-Free Graphs}
\author{Xinheng Lin\\[1.4ex]
  \normalsize Center for Discrete Mathematics, Fuzhou University\\[-0.1ex]
  \normalsize Fuzhou, P.\ R.\ China}
\date{August 6, 2026}

\begin{document}

\setlength{\abovedisplayskip}{6pt plus 2pt minus 2pt}
\setlength{\belowdisplayskip}{6pt plus 2pt minus 2pt}
\setlength{\abovedisplayshortskip}{0pt plus 2pt}
\setlength{\belowdisplayshortskip}{3pt plus 2pt minus 1pt}

\maketitle

\begin{abstract}
Delcourt and Postle reduced the Linear Hadwiger Conjecture to coloring
$K_t$-minor-free graphs on $O(t\log^4 t)$ vertices. An important theorem in
their proof process asserts that every sufficiently dense $K_t$-minor-free
graph contains a small, highly connected subgraph. In this paper, we show that
such a subgraph can be chosen to be smaller. More precisely, there
exists an integer constant $C\geq 1$ such that, for all integers $t\geq 3$ and
$k\geq t$, every $K_t$-minor-free graph $G$ with $d(G)\geq Ck$ contains a
nonempty $k$-connected subgraph $H$ satisfying
\[
  v(H)\leq C^2t\log t.
\]
Thus the structural bound improves from $O(t\log^3 t)$ to $O(t\log t)$. We
also give a probabilistic construction showing that the $t\log t$ bound on
$v(H)$ is best possible up to a constant factor. Consequently, the graphs
occurring in the reduction have order $O(t\log^2 t)$ rather than
$O(t\log^4 t)$.

\end{abstract}

\footnotetext[1]{
Email: \href{mailto:2500310007@fzu.edu.cn}{2500310007@fzu.edu.cn}}

\section{Introduction}

All graphs in this paper are finite and simple. For a graph $G$, we denote its
chromatic number by $\chi(G)$. For graphs $H$ and $G$, we write
$H\preccurlyeq_{\mathrm m}G$ if $H$ is a minor of $G$. We denote by $K_t$ the
complete graph on $t$ vertices, and say that $G$ is $K_t$-minor-free if
$K_t\not\preccurlyeq_{\mathrm m}G$.

Reed and Seymour~\cite{reed-seymour} raised the following linear weakening of
Hadwiger's conjecture.

\begin{conjecture}[Linear Hadwiger Conjecture~\cite{reed-seymour}]\label{conj:linear-hadwiger}
There exists an integer constant $C>0$ such that, for every integer $t\geq1$,
every $K_t$-minor-free graph $G$ satisfies $\chi(G)\leq Ct$.
\end{conjecture}

In 2025, Delcourt and Postle~\cite{delcourt-postle} reduced
Conjecture~\ref{conj:linear-hadwiger} to coloring small $K_t$-minor-free
graphs.

\begin{theorem}[Delcourt--Postle~\cite{delcourt-postle}]\label{thm:reduction}
There exists an integer constant $C_{1.2}\geq 1$ such that the following
holds. If, for every integer $t\geq 3$, every $K_t$-minor-free graph $H$ with
\[
  v(H)\leq C_{1.2}t\log^4 t
\]
satisfies $\chi(H)\leq C_{1.2}t$, then, for every integer $t\geq 3$, every
$K_t$-minor-free graph $G$ satisfies $\chi(G)\leq C_{1.2}^2t$.
\end{theorem}

One of the key steps in establishing Theorem~\ref{thm:reduction} is to prove
the structural result stated in \cite[Theorem~2.3]{delcourt-postle}.

\begin{theorem}[Delcourt--Postle~\cite{delcourt-postle}]\label{thm:dp-structural}
There exists an integer constant $C_{1.3}\geq 1$ such that the following
holds. Let $t\geq 1$ and $k\geq t$ be integers. If $G$ is a
$K_t$-minor-free graph with $d(G)\geq C_{1.3}k$, then $G$ contains a
nonempty $k$-connected subgraph $H$ such that
\[
  v(H)\leq C_{1.3}^2t\log^3 t.
\]
\end{theorem}

In~\cite{delcourt-postle}, Delcourt and Postle first choose a collection
$H_1,\ldots,H_m$ of pairwise vertex-disjoint connected subgraphs, with $m$
as large as possible, and contract each $H_i$ to a single vertex. They then
discard vertices of degree larger than a constant multiple of $k\log t$ and
obtain a nonempty connected subgraph $U$ of order
\[
  O\!\left(\log^2 t\left(\frac{t}{k}\right)^2\right).
\]
Using this subgraph, they choose a set $R$ in its neighborhood such that
$G[R]$ is dense and has order
\[
  O\!\left(k\log t\cdot\log^2 t\left(\frac{t}{k}\right)^2\right)
  =O\!\left(\frac{t^2}{k}(\log t)^3\right).
\]
Our strategy follows the same neighborhood construction, but controls the
size of the resulting neighborhood directly through the sum
$\sum_{u\in U}\deg_G(u)$, rather than controlling both $|U|$
and this sum as in~\cite{delcourt-postle}. Thus no truncation of high-degree
vertices is needed,
saving one factor of $\log t$. Moreover, by refining the set $Y$ of vertices
with many neighbors among the contracted blocks, we can choose $U$ so that
this sum is smaller, saving a second factor of $\log t$.

The resulting estimate is the following.

\begin{theorem}\label{thm:main}
There exists an integer constant $C_{1.4}\geq 1$ such that the following
holds. Let $t\geq 3$ and $k\geq t$ be integers. If $G$ is a
$K_t$-minor-free graph with $d(G)\geq C_{1.4}k$, then $G$ contains a
nonempty $k$-connected subgraph $H$ such that
\[
  v(H)\leq C_{1.4}^2t\log t.
\]
\end{theorem}

Replacing Theorem~\ref{thm:dp-structural} with Theorem~\ref{thm:main} in the
reduction of~\cite{delcourt-postle} improves the order bound in
Theorem~\ref{thm:reduction} from $O(t\log^4 t)$ to $O(t\log^2 t)$.
Indeed, a direct check in the proofs of~\cite[Corollary~6.1, Lemma~6.4,
and Theorem~7.1]{delcourt-postle} shows that the construction uses
$O((\sqrt{\log t})^2)=O(\log t)$ such subgraphs, each now of order
$O(t\log t)$; hence their total order is $O(t\log^2 t)$, and the remaining
arguments are unchanged after adjusting the constants.

The order bound in Theorem~\ref{thm:main} is best possible up to its constant
factor. In fact, the following result shows this already in the case $k=t$.

\begin{theorem}\label{thm:lower-bound}
For every fixed integer $A\geq2$, there is an integer $t_A$ such that, for
every integer $t\geq t_A$, there exists a $K_t$-minor-free graph $G$ with
$d(G)\geq At$ such that every $t$-connected
subgraph $H$ of $G$ satisfies
\[
  v(H)>\frac{1}{7200\mathrm e A^2}t\log t.
\]
\end{theorem}

We prove Theorem~\ref{thm:lower-bound} in Section~\ref{sec:lower-bound} by a
probabilistic construction.

\section{Preliminaries}

All logarithms in this paper are natural. We use largely standard
graph-theoretic notation. For a graph $G$, we denote by $V(G)$ and $E(G)$ its
vertex set and edge set, respectively, and write $v(G):=|V(G)|$ and
$e(G):=|E(G)|$. If $G$ is nonempty, its density is $d(G):=e(G)/v(G)$. The
neighborhood and degree of a vertex $v\in V(G)$ are denoted by $N_G(v)$ and
$\deg_G(v)$, respectively. We denote the minimum degree of $G$ by $\delta(G)$
and its vertex connectivity by $\kappa(G)$.

For a set $X\subseteq V(G)$, let $G[X]$ denote the subgraph of $G$ induced by
$X$. If $A,B\subseteq V(G)$ are disjoint, then $G(A,B)$ denotes the bipartite
subgraph of $G$ with vertex set $A\cup B$ and edge set
$\{uv\in E(G):u\in A,\ v\in B\}$, and we write
$e_G(A,B):=|E(G(A,B))|$. For a nonnegative integer $m$, let
$[m]:=\{1,\ldots,m\}$, with $[0]:=\varnothing$.

We use the standard branch-set characterization of a complete minor. Thus
$K_t\preccurlyeq_{\mathrm m}G$ if and only if there are pairwise disjoint
nonempty sets $B_1,\ldots,B_t\subseteq V(G)$ such that every $G[B_i]$ is
connected and $e_G(B_i,B_j)>0$ whenever $i\neq j$.

If $A\subseteq V(G)$ is nonempty and $G[A]$ is connected, then $G/A$ denotes
the simple graph obtained from $G$ by contracting $A$ to a single vertex,
deleting loops, and suppressing parallel edges. More generally, if
$H_1,\ldots,H_m$ are pairwise vertex-disjoint connected subgraphs of $G$, then
$G/(H_1,\ldots,H_m)$ denotes the simple graph obtained by contracting each
$H_i$ to a single vertex. For an edge $uv\in E(G)$, we use $G/uv$ as shorthand
for $G/\{u,v\}$.

For a positive integer $n$ and $p\in[0,1]$, the binomial random graph
$G(n,p)$ has vertex set $[n]$, with every pair of vertices joined
independently with probability $p$. We write $\operatorname{Bin}(M,p)$ for a
binomial random variable with parameters $M$ and $p$.

The proof requires the asymmetric bipartite estimate of Norin and
Postle~\cite{norin-postle}, in the form recorded by Delcourt and
Postle~\cite{delcourt-postle}.

\begin{theorem}[Norin--Postle~\cite{norin-postle}]\label{thm:bipartite}
There exists an integer constant $C_{2.1}\geq 1$ such that, for every integer
$t\geq 3$ and every $K_t$-minor-free bipartite graph $B$ with bipartition
$(P,Q)$,
\[
  e(B)\leq C_{2.1}t\sqrt{\log t}\sqrt{|P||Q|}
  +(t-2)(|P|+|Q|).
\]
\end{theorem}

The dense-core lemma from~\cite[Lemma~4.2]{delcourt-postle} is recorded next.

\begin{lemma}[Delcourt--Postle~\cite{delcourt-postle}]\label{lem:dense-core}
Let $r>2$ and $\delta>0$, let $G$ be a graph, and let $S\subseteq V(G)$. If
\[
  (r-2)e(G[S])>(r-1)\delta|S|+e_G(S,V(G)\setminus S),
\]
then there is a nonempty set $S'\subseteq S$ such that
\begin{enumerate}
  \item $\delta(G[S'])\geq\delta$; and
  \item for every $v\in S'$,
  \[
    |N_G(v)\cap S'|\geq\frac{\deg_G(v)}{r}.
  \]
\end{enumerate}
\end{lemma}

The remaining ingredients are the Kostochka--Thomason density
bound~\cite{kostochka,thomason}, in the explicit form of
\cite[Theorem~4.3]{delcourt-postle}, and Mader's connectivity
theorem~\cite{mader}.

\begin{theorem}[Kostochka--Thomason~\cite{kostochka,thomason}]\label{thm:kt-density}
Let $t\geq 2$ be an integer. If a nonempty graph $G$ satisfies
\[
  d(G)\geq 30t\sqrt{\log t},
\]
then $K_t\preccurlyeq_{\mathrm m}G$.
\end{theorem}

\begin{theorem}[Mader~\cite{mader}]\label{thm:mader}
Every nonempty graph $G$ contains a nonempty subgraph $H$ such that
\[
  \kappa(H)\geq\frac{d(G)}{2}.
\]
\end{theorem}

\section{Proof of the main theorem}

The proof follows the framework of Delcourt and
Postle~\cite[proof of Theorem~2.3]{delcourt-postle}.

\begin{proof}[Proof of Theorem~\ref{thm:main}]
Put $C_{1.4}:=\max\{30300,32C_{2.1}^2\}$. Let $t\geq3$ and $k\geq t$ be integers,
let $G$ be a $K_t$-minor-free graph with $d(G)\geq C_{1.4}k$, and put
$d:=C_{1.4}k$. By deleting edges if necessary, we may assume
that $e(G)=d\,v(G)$. Since $G$ is $K_t$-minor-free,
Theorem~\ref{thm:kt-density} gives $d<30t\sqrt{\log t}$. Set
$Q:=8C_{2.1}^2d\log t\,(t/k)^2$.

For $A\subseteq V(G)$, write
$w_G(A):=\sum_{v\in A}\deg_G(v)$. For every graph $J$ occurring below,
we use $w_J$ analogously. Thus $w_G(V(G))=2e(G)=2d\,v(G)$.

Let $H_1,\ldots,H_m$ be nonempty pairwise vertex-disjoint connected subgraphs
of $G$ such that $w_G(V(H_i))\geq Q$ for every $i\in[m]$ and, after contracting
each $H_i$ to a new vertex $x_i$, the resulting graph $G'$ satisfies
\[
  e(G)-e(G')\leq\frac{1}{100}\sum_{i=1}^m w_G(V(H_i)),
\]
with $m$ chosen as large as possible. The empty collection is admissible, so
such a collection exists. Since $G'$ is a minor of $G$, it is
$K_t$-minor-free. Put $X:=\{x_i:i\in[m]\}$ and
$O:=V(G')\setminus X$. Summing the degrees over these pairwise disjoint
vertex sets gives
\[
  |X|Q=mQ\leq\sum_{i=1}^m w_G(V(H_i))\leq w_G(V(G))=2d\,v(G),
\]
and hence
\[
  |X|\leq\frac{v(G)}{4C_{2.1}^2\log t}\left(\frac{k}{t}\right)^2.
\]
The defining contraction inequality gives
\[
  e(G')\geq d\,v(G)-\frac{1}{100}\sum_{i=1}^m w_G(V(H_i))
  \geq\frac{49}{50}d\,v(G).
\]

\begin{claim}\label{clm:no-large-contractible-set}
There is no nonempty set $A\subseteq O$ such that $G'[A]$ is connected,
$w_{G'}(A)\geq Q$, and
$e(G')-e(G'/A)\leq w_{G'}(A)/100$. In particular, every vertex $v\in O$
satisfies $\deg_{G'}(v)<Q$.
\end{claim}

\begin{proof}
Since the contractions used to form $G'$ do not change the adjacency relations among vertices of $O$, we have
$G'[O]=G[O]$. The first assertion follows from the maximality of $m$. Indeed,
otherwise
$w_G(A)\geq w_{G'}(A)\geq Q$, and
\[
  \begin{aligned}
  e(G)-e(G'/A)
  &=e(G)-e(G')+e(G')-e(G'/A)\\
  &\leq\frac{1}{100}\sum_{i=1}^m w_G(V(H_i))
    +\frac{1}{100}w_{G'}(A)\\
  &\leq\frac{1}{100}
    \left(\sum_{i=1}^m w_G(V(H_i))+w_G(A)\right),
  \end{aligned}
\]
so $A$ could be added to the collection $H_1,\ldots,H_m$, a contradiction.
In particular, every vertex $v\in O$ satisfies $\deg_{G'}(v)<Q$. Otherwise,
$A=\{v\}$ would contradict the first assertion.
\end{proof}

Define
\[
  Y:=\left\{v\in O:
  |N_{G'}(v)\cap X|\geq\frac{1}{200}\deg_{G'}(v)\right\},
\]
and put $T:=X\cup Y$ and $S:=V(G')\setminus T$.

\begin{claim}\label{clm:exceptional}
 $e(G'[T])+e_{G'}(S,T)<d\,v(G)/100$.
\end{claim}

\begin{proof}
Applying Theorem~\ref{thm:bipartite} to the bipartite graph $G'(X,O)$ and
using $|O|\leq v(G)$ gives
\begin{align*}
  e_{G'}(X,O)
  &\leq C_{2.1}t\sqrt{\log t}\sqrt{|X||O|}+(t-2)(|X|+|O|)\\
  &\leq C_{2.1}t\sqrt{\log t}
  \sqrt{\frac{v(G)}{4C_{2.1}^2\log t}
  \left(\frac{k}{t}\right)^2v(G)}+t\,v(G)\\
  &\leq \frac{1}{2}k\,v(G)+k\,v(G)=\frac{3}{2}k\,v(G),
\end{align*}
By the definition of $Y$,
\[
  \frac{1}{200}\sum_{v\in Y}\deg_{G'}(v)
  \leq e_{G'}(X,Y)\leq e_{G'}(X,O),
\]
and hence
\[
  \sum_{v\in Y}\deg_{G'}(v)\leq300k\,v(G).
\]

If $X\neq\varnothing$, then Theorem~\ref{thm:kt-density} and the bound on
$|X|$ give
\[
  e(G'[X])<30t\sqrt{\log t}\,|X|
  \leq\frac{15v(G)k^2}{2C_{2.1}^2t\sqrt{\log t}}.
\]
Since $d=C_{1.4}k<30t\sqrt{\log t}$, we have
$k/(t\sqrt{\log t})<30/C_{1.4}$. Since $C_{1.4}>225$, it follows that
$e(G'[X])<225k\,v(G)/(C_{2.1}^2C_{1.4})<k\,v(G)$. The same conclusion is immediate when
$X=\varnothing$. Thus
\[
  e(G'[T])+e_{G'}(S,T)
  \leq e(G'[X])+e_{G'}(X,O)+\sum_{v\in Y}\deg_{G'}(v)
  <303k\,v(G)\leq\frac{d\,v(G)}{100},
\]
where the final inequality follows from $d=C_{1.4}k$ and
$C_{1.4}\geq30300$.
\end{proof}

By Claim~\ref{clm:exceptional},
$e(G'[S])=e(G')-e(G'[T])-e_{G'}(S,T)>97d\,v(G)/100$. Moreover,
\[
  \frac{4d}{5}|S|+e_{G'}(S,V(G')\setminus S)
  \leq\frac{4d}{5}v(G)+\frac{d\,v(G)}{100}
  =\frac{81}{100}d\,v(G)<e(G'[S]).
\]
Lemma~\ref{lem:dense-core}, applied with $r=3$ and $\delta=2d/5$, now yields
a nonempty set $S'\subseteq S$ such that $\delta(G'[S'])\geq2d/5$ and
$|N_{G'}(v)\cap S'|\geq\deg_{G'}(v)/3$ for every $v\in S'$.

Every connected component $K$ of $G'[S']$ satisfies
$w_{G'}(V(K))>Q$. Indeed, $\delta(G'[S'])\geq2d/5$ implies
$v(K)\geq2d/5+1$, and hence
$w_{G'}(V(K))\geq(2d/5)v(K)>4d^2/25$. On the other hand, since
$\log t\leq t/2\leq k/2$ for $t\geq3$ and $C_{1.4}\geq32C_{2.1}^2$,
\[
  Q:=8C_{2.1}^2d\log t\left(\frac{t}{k}\right)^2
  \leq4C_{2.1}^2dk=\frac{4C_{2.1}^2}{C_{1.4}}d^2
  \leq\frac{d^2}{8}<\frac{4d^2}{25}.
\]

Choose a nonempty set $U\subseteq S'$ such that $G'[U]$ is connected,
$w_{G'}(U)<Q$, and $e(G')-e(G'/U)\leq w_{G'}(U)/100$, with
$|U|$ as large as possible. Such a set exists because every singleton in $S'$ satisfies the
conditions above by Claim~\ref{clm:no-large-contractible-set}. Let $K_U$ be the
component of $G'[S']$ containing $U$. The preceding bound gives
$w_{G'}(V(K_U))>Q>w_{G'}(U)$, so $U$ is a proper subset of $V(K_U)$ and some
vertex of $S'\setminus U$ has a neighbor in $U$.

\begin{claim}\label{clm:extension}
For every $z\in S'\setminus U$ with $N_{G'}(z)\cap U\neq\varnothing$, we have
$e(G')-e(G'/(U\cup\{z\}))>w_{G'}(U\cup\{z\})/100$.
\end{claim}

\begin{proof}
Suppose the asserted inequality fails. Since $z$ has a neighbor in $U$,
$G'[U\cup\{z\}]$ is connected. If $w_{G'}(U\cup\{z\})<Q$, then
$U\cup\{z\}$ satisfies the defining conditions for $U$ and has larger
order, a contradiction. We may therefore assume that
$w_{G'}(U\cup\{z\})\geq Q$. Since all vertices of $U\cup\{z\}$ belong to
$O$, this contradicts Claim~\ref{clm:no-large-contractible-set}, applied with
$A:=U\cup\{z\}$.
\end{proof}

Let $G'':=G'/U$, and let $x$ be the vertex obtained by contracting $U$. Set
$R:=N_{G''}(x)\setminus X$ and
$R':=N_{G''}(x)\cap(S'\setminus U)$. A direct count of the edges lost
when $U$ is contracted gives
\begin{align*}
  \sum_{u\in U}|N_{G'}(u)\cap S'|-|R'|
  &=2e(G'[U])
    +\sum_{\substack{y\in S'\setminus U\\ |N_{G'}(y)\cap U|\geq1}}
      \bigl(|N_{G'}(y)\cap U|-1\bigr)\\
  &\leq 2\bigl(e(G')-e(G'/U)\bigr)\leq\frac{1}{50}w_{G'}(U).
\end{align*}
Since $|N_{G'}(v)\cap S'|\geq\deg_{G'}(v)/3$ for every $v\in S'$, it follows
that
\[
  \sum_{u\in U}|N_{G'}(u)\cap S'|\geq\frac{1}{3}w_{G'}(U).
\]
Consequently,
\[
  |R'|\geq\left(\frac{1}{3}-\frac{1}{50}\right)w_{G'}(U)
  >\frac{1}{4}w_{G'}(U)\geq\frac{|R|}{4}.
\]
In particular, $R'$ and $R$ are nonempty.

Fix $z\in R'$. Contracting $U$ and then the edge $xz$ has the same effect as
contracting $U\cup\{z\}$ in $G'$. By Claim~\ref{clm:extension} and the
defining property of $U$,
\begin{align*}
  e(G'')-e(G''/xz)
  &=\bigl(e(G')-e(G'/(U\cup\{z\}))\bigr)
    -\bigl(e(G')-e(G'/U)\bigr)\\
  &>\frac{1}{100}\bigl(w_{G'}(U\cup\{z\})-w_{G'}(U)\bigr)\\
  &=\frac{1}{100}\deg_{G'}(z).
\end{align*}
Contracting an edge in a simple graph deletes the edge itself and one
additional edge for each common neighbor of its endpoints. Thus $x$ and $z$
have more than $\deg_{G'}(z)/100-1$ common neighbors in $G''$. Moreover,
$z\in S'\subseteq S$ implies $z\notin Y$, and hence
$|N_{G'}(z)\cap X|<\deg_{G'}(z)/200$. Contracting $U$ does not change the
neighbors of $z$ in $X$. Therefore more than
$\deg_{G'}(z)/200-1$ common neighbors of $x$ and $z$ lie outside $X$, and
every such vertex belongs to $R$. Since $\deg_{G'}(z)\geq2d/5$ and
$d=C_{1.4}k\geq1000$, we have
\[
  |N_{G''}(z)\cap R|
  >\frac{\deg_{G'}(z)}{200}-1
  \geq\frac{d}{500}-1
  \geq\frac{d}{1000}.
\]

Summing over $z\in R'$ gives
\[
  2e(G''[R])\geq\sum_{z\in R'}|N_{G''}(z)\cap R|
  \geq\frac{d}{1000}|R'|>\frac{d}{4000}|R|,
\]
and hence $d(G''[R])>d/8000$. No vertex of $R$ is involved in the contractions
used to form $G''$, so $G''[R]=G[R]$. By
Theorem~\ref{thm:mader},
$G[R]$ contains a nonempty subgraph $H$ satisfying
\[
  \kappa(H)\geq\frac{d(G[R])}{2}>\frac{d}{16000}
  =\frac{C_{1.4}}{16000}k\geq k.
\]
Finally,
\begin{align*}
  v(H)
  &\leq |R|\leq w_{G'}(U)<Q
   =8C_{2.1}^2C_{1.4}k\log t\left(\frac{t}{k}\right)^2
   \leq C_{1.4}^2t\log t.
\end{align*}
where the final inequality uses $k\geq t$ and
$C_{1.4}\geq32C_{2.1}^2$, which in particular gives
$C_{1.4}\geq8C_{2.1}^2$.
Thus $H$ has the required connectivity and order.
\end{proof}

\section{A matching lower bound}\label{sec:lower-bound}

We now prove the matching lower bound stated in Theorem~\ref{thm:lower-bound}.

\begin{proof}[Proof of Theorem~\ref{thm:lower-bound}]
Fix an integer $A\geq2$, and let $t$ be sufficiently large in terms of $A$.
Throughout the proof, all asymptotic notation is with respect to $t\to\infty$,
with $A$ fixed. Put
\[
  n:=\left\lceil\frac{t\log t}{100A}\right\rceil,
  \qquad
  p:=\frac{9At}{n},
  \qquad
  G\sim G(n,p).
\]
Thus $np=9At$ and $p\sim900A^2/\log t$. In particular, $p\leq1/2$ for all
sufficiently large $t$.

We first estimate the density of $G$. Let $X:=e(G)$ and
$\mu:=\mathbb E X=p\binom{n}{2}=9At(n-1)/2$. For all sufficiently large $t$,
$Atn\leq(1-\frac{1}{10})\mu$. By the Chernoff lower-tail bound,
\[
  \mathbb P\bigl(d(G)<At\bigr)
  \leq\mathbb P\left(X<\left(1-\frac{1}{10}\right)\mu\right)
  \leq\exp\left(-\frac{(1/10)^2\mu}{2}\right)=o(1).
\]

We next use a standard first-moment argument for complete-minor models to
estimate the probability that $G$ contains a $K_t$ minor. This argument goes
back to Bollob\'as, Catlin and Erd\H{o}s~\cite{bollobas-catlin-erdos}; see also
Fountoulakis, K\"uhn and Osthus~\cite[Section~2]{fountoulakis-kuhn-osthus}.
Suppose
that $B_1,\ldots,B_t$ are the branch sets of such a minor. Since the branch
sets are pairwise disjoint, at least $q:=\lfloor t/2\rfloor$ of them have
order at most $2n/t$. Choose and order $q$ such branch sets, and ignore their
internal connectivity.

Fix pairwise disjoint nonempty sets $B_1,\ldots,B_q$, each of order at most
$2n/t$. Since $p\leq1/2$, the inequality $1-x\geq\exp(-2x)$ for
$0\leq x\leq1/2$ gives, for $i<j$,
\[
  \mathbb P\bigl(e_G(B_i,B_j)=0\bigr)
  =(1-p)^{|B_i||B_j|}
  \geq\exp\!\left(-\frac{8pn^2}{t^2}\right)
  =\exp\!\left(-\frac{72An}{t}\right).
\]
Since $n\leq t\log t/(100A)+1$, we have
\[
  \frac{72An}{t}
  \leq\frac{18}{25}\log t+\frac{72A}{t}
  \leq\frac{4}{5}\log t
\]
for all sufficiently large $t$. Hence $B_i$ and $B_j$ are joined by an edge
with probability at most $1-t^{-4/5}$.

Since $\binom{q}{2}\geq t^2/16$ for all sufficiently large $t$, the inequality
$1-x\leq\exp(-x)$ gives
\[
  \mathbb P\bigl(e_G(B_i,B_j)>0\text{ for all }i<j\bigr)
  \leq(1-t^{-4/5})^{\binom{q}{2}}
  \leq\exp(-t^{6/5}/16).
\]
The number of ordered pairwise disjoint families $(B_1,\ldots,B_q)$ is at
most $(q+1)^n$: assign label $i$ to each vertex in $B_i$, for $i\in[q]$, and
label $0$ to every vertex outside $\bigcup_{i=1}^q B_i$. By the union bound,
\[
  \mathbb P(K_t\preccurlyeq_{\mathrm m}G)
  \leq(q+1)^n\exp(-t^{6/5}/16)
  \leq\exp\!\left(n\log(t+1)-t^{6/5}/16\right)=o(1),
\]
where the last equality follows from $n=O_A(t\log t)$.

It remains to rule out small $t$-connected subgraphs. Fix
$S\subseteq V(G)$ with $|S|=s$, where
\[
  t+1\leq s\leq\frac{n}{72\mathrm e A},
\]
and put $Y:=e(G[S])$ and $r:=\lceil ts/2\rceil$. Then
$Y\sim\operatorname{Bin}(\binom{s}{2},p)$. Since
$s\leq n/(72\mathrm e A)$, we have $r\geq\mathbb E Y$. By the standard
binomial upper-tail bound and again using $s\leq n/(72\mathrm e A)$,
\[
  \mathbb P(Y\geq ts/2)
  \leq\left(\frac{\mathrm e\binom{s}{2}p}{r}\right)^r
  \leq\left(\frac{\mathrm e ps}{t}\right)^r
  \leq8^{-r}
  \leq8^{-ts/2}.
\]
Since $n\leq t\log t/(100A)+1$ and $s\geq t+1>t$, we have
$n/s\leq\log t/(100A)+1/t$. Thus the standard estimate
$\binom{n}{s}\leq(\mathrm e n/s)^s$ gives, for all sufficiently large $t$,
\[
\begin{aligned}
  \binom{n}{s}
  &\leq\left(\frac{\mathrm e n}{s}\right)^s
  =\exp\left(s\log\left(\frac{\mathrm e n}{s}\right)\right)
  \leq\exp\left(s\log\left(\mathrm e
      \left(\frac{\log t}{100A}+\frac{1}{t}\right)\right)\right) \\
  &\leq\exp\left(\frac{ts}{4}\log 8\right)=8^{ts/4}.
\end{aligned}
\]
Here the last inequality holds because
$\log\bigl(\mathrm e(\log t/(100A)+1/t)\bigr)=O_A(\log\log t)$,
which is at most $(t/4)\log 8$ for all sufficiently large $t$.
A union bound over $s$ now gives
\begin{align*}
  &\mathbb P\left(\text{there exists a set }S\subseteq V(G)\text{ with }
    t+1\leq|S|\leq\frac{n}{72\mathrm e A}
    \text{ and }e(G[S])\geq\frac{t|S|}{2}\right)\\
  &\leq
    \sum_{s=t+1}^{\left\lfloor n/(72\mathrm e A)\right\rfloor}
    \binom{n}{s}\mathbb P(Y\geq ts/2)
    \leq\sum_{s=t+1}^{\infty}8^{-ts/4}
    =\frac{8^{-t(t+1)/4}}{1-8^{-t/4}}=o(1).
\end{align*}

The three failure probabilities above are $o(1)$. Hence, by the union bound,
with probability $1-o(1)$, the graph $G$ has
density at least $At$, is $K_t$-minor-free, and has no set
$S\subseteq V(G)$ with $t+1\leq|S|\leq n/(72\mathrm e A)$ and
$e(G[S])\geq t|S|/2$. Fix a graph $G$ with these properties. Let $H$ be any
$t$-connected subgraph of $G$ and put $S:=V(H)$.
Then $|S|\geq t+1$, and the minimum degree of $H$ is at least $t$. Hence
\[
  e(G[S])\geq e(H)\geq\frac{t|S|}{2}.
\]
By the choice of $G$, it follows that
\[
  v(H)>\frac{n}{72\mathrm e A}
  \geq\frac{1}{7200\mathrm e A^2}t\log t,
\]
as required.
\end{proof}

For every fixed integer $A\geq2$, Theorem~\ref{thm:lower-bound} gives
$K_t$-minor-free graphs of density at least $At$ in which every
$t$-connected subgraph has order at least a constant multiple of $t\log t$.
Thus, in the case $k=t$, the order bound in Theorem~\ref{thm:main} is best
possible up to a constant factor.

\section{Acknowledgments}

This research was partially supported by the Natural Science Foundation of
Fujian Province (Grant No. 2025J01486).

\medskip
\noindent\textbf{Statement of AI Use.}
The central proof idea in this manuscript was developed with the assistance of
GPT 5.6 Sol. The author subsequently refined the idea and wrote
the manuscript, using the same model to assist with language polishing.

\end{document}